\documentclass[10pt]{amsart}

\newtheorem{theorem}{Theorem}[section]
\newtheorem{lemma}[theorem]{Lemma}
\newtheorem{corollary}[theorem]{Corollary}

\theoremstyle{definition}
\newtheorem{definition}[theorem]{Definition}
\newtheorem{example}[theorem]{Example}
\newtheorem{convention}[theorem]{Convention}

\newtheorem{remark}[theorem]{Remark}
\newtheorem{mainfact}[theorem]{Main Fact}
\numberwithin{equation}{section}

\newcommand{\arxiv}[1]{\href{http://arxiv.org/abs/#1}{\texttt{arXiv:#1}}}

\usepackage[vcentermath, enableskew]{youngtab}
\usepackage{mathtools}
\usepackage[colorlinks=true, pdfstartview=FitV, linkcolor=blue, citecolor=blue, urlcolor=blue]{hyperref}
\usepackage{enumerate}
\usepackage{tikz}
\usepackage{wasysym}
\usepackage{amssymb}
\usepackage{multicol}
\usepackage{amsmath,amscd}

\title{Unified framework for tableau models of Grothendieck polynomials}

\author[G.~Hawkes]{Graham Hawkes}
\address[G. Hawkes]{Department of Mathematics, Ben Gurion University of the Negev, 1 David Ben Gurion Blvd., Be'er Sheva, Israel}
\email{ghawkes1217@gmail.com}


\begin{document}

\begin{abstract}
We give combinatorial proofs of two types of duality for Grothendieck polynomials by constructing a unified combinatorial framework incorporating set-valued tableaux, musltiset-valued tableaux, reverse plane partitions and valued-set tableaux.   Importantly, our proofs extend to proofs of these dualities for the refined Grothendieck polynomials. The second of these dualities was formerly unknown for the refined case.
\end{abstract}

\maketitle


\section{Introduction}

 Lascoux and Sch\"utzenberger~\cite{LS82} introduced Grothendieck polynomials to represent the K-theory ring of the Grassmannian.  Fomin and Kirillov~\cite{FK94} later  initiated the study of stable Grothendieck functions. The subset of stable Grothendieck functions corresponding to Grassmannian permutations can be indexed by partitions and are called symmetric Grothendieck polynomials. They are Schur positive and form a linearly independent set in the space of symmetric functions of any degree. In particular, the lowest degree term of a symmetric Grothendieck polynomial is the Schur function indexed by the same partition.

We now remind the reader of two very important operations involving the ring of symmetric functions.  The first is the involution, $\omega$, that interchanges elementary and homogeneous symmetric functions.  The second is the Hall inner product, $\langle ,  \rangle$, which is the bilinear form defined by declaring the homogeneous and monomial symmetric bases of the ring of symmetric functions to be orthonormal.  Aside from being incredibly useful tools for proving essential properties of symmetric functions  these operations have an amazing relationship to Schur functions.  First, it can be shown that the involution $\omega$ sends the Schur function $s_{\lambda}$ to $s_{\lambda'}$, the Schur function indexed by the transpose of $\lambda$.  Second, the under the Hall inner product we can show that the Schur functions form an orthonormal basis for the ring of symmetric functions, meaning that $\langle s_{\lambda},s_{\mu}\rangle = \delta_{\lambda,\mu}$.

Since the symmetric Grothendieck polynomial is a generalization of the Schur function, it is natural to ask (1) how the involution $\omega$ affects symmetric Grothendieck polynomials, and (2) whether there is a set of polynomials orthonormal under the Hall inner product to the symmetric Grothendieck polynomials. These considerations led to the establishment of the four following versions of the symmetric Grothendieck polynomials and the study of their properties in relation to the ``big Hopf algebra of Multisymmetric functions" described in section 9 of \cite{LP07} as well as to the invention of the following combinatorial models to represent them.

\begin{enumerate}
\item[(1A)] Symmetric Grothendieck polynomials using set-valued tableaux \cite{Buch02},
\item [(1B)] Weak symmetric Grothendieck polynomials using multiset-valued tableaux \cite{LP07},
\item [(2A)] Dual symmetric Grothendieck polynomials using reverse plane partitions \cite{LP07},
\item [(2B)]Dual weak symmetric Grothendieck polynomials using valued-set tableaux \cite{LP07}.
\end{enumerate}

These polynomials have the following relationships: (1A) goes to (1B) under $\omega$ and (2A) goes to (2B) under $\omega$ while (1A) is dual under $\langle,\rangle$ to (2A) and (1B) is dual under $\langle,\rangle$ to (2B). In other words we have the following diagram:

\begin{center}
$\begin{CD}
\mathrm{1A} @>\omega>> \mathrm{1B}\\
@VV\langle,\rangle V @VV\langle,\rangle V\\
\mathrm{2A} @>\omega>> \mathrm{2B}
\end{CD}$
\end{center}

Crystal analyses of these objects such as the analysis of reverse plane partitions appearing in \cite{GAL17}, the analysis of set-valued tableaux appearing in \cite{MPS18}, and the analysis of multiset-valued tableaux and valued-set tableaux appearing in \cite{HS20}, point toward a meaningful refinement of these polynomials using an additional parameter (in this paper $\mathbf{z}$). In particular, if we analyze the crystal structures given in each one of these three papers, we find that the connected components of the crystals respect this refinement.  In other words, the crystal structures given not only prove the Schur positivity of the polynomials considered in the papers, but also, the Schur positivity of the corresponding refined versions.  Although this is not explicitly stated in these papers, it is not difficult to deduce this fact from the constructions given therein.  For additional background on the refined Grothendieck polynomials the reader is also suggested to see definition 3.2, Theorem 3.3, and remark 3.9 of \cite{CP19}. Each of these refinements appears naturally in the underlying combinatorial object and denoting the new refined versions as \textbf{1A}, \textbf{1B}, \textbf{2A}, and \textbf{2B} we still have:

\begin{center}
$\begin{CD}
\mathbf{1A} @>\omega>> \mathbf{1B}\\
@VV\langle,\rangle V @VV\langle,\rangle V\\
\mathbf{2A} @>\omega>> \mathbf{2B}
\end{CD}$
\end{center}

This paper provides a combinatorial explanation of all these arrows, which was previously an open question. In order to do this, we begin by realizing each of the four combinatorial models as a particular instance of a certain type of tableau. Then we use a standard RSK argument and a standard jeu de taquin argument to find out what the fundamental combinatorial facts we need to prove are. As it turns out, a single combinatorial fact (main fact \ref{mf1}) explains both horizontal arrows, and another fact (main fact \ref{mf2}) explains both vertical arrows.

We remark that it may be worthwhile to investigate how the combinatorial results of this paper could be used to understand the significance of the refinements of the four polynomials studied here to the big Hopf algebra of Multisymmetric functions of \cite{LP07}.

\subsection{Primed Tableaux}

\begin{definition}
Consider the alphabet $\{1'<1<2'<2<3'<3<\cdots\}$. An \emph{overfull tableau} of shape $\lambda$, or an element of $\mathrm{OT}(\lambda)$, is a filling of a Young diagram of shape $\lambda$ such that
\begin{itemize}
\item Each box is nonempty and contains a set from $\{1',2',\ldots\}$ and a multiset from $\{1,2,\ldots\}$.
\item Suppose there is an $a$ in box $A$ and a $b$ in box $B$. Suppose box $B$ lies immediately to the right of box $A$. Then $a<b$ or else $a=b$ and both are unprimed numbers.
\item Suppose there is an $a$ in box $A$ and a $c$ in box $C$. Suppose box $C$ lies immediately below box $A$. Then $a<c$ or else $a=c$ and both are primed numbers.
\end{itemize}

An \emph{underfull tableau} of shape $\lambda$, or an element of $\mathrm{UT}(\lambda)$, is a filling of a Young diagram of shape $\lambda$ such that
\begin{itemize}
\item Each box is either empty or contains one number from either the set $\{1',2',\ldots\}$ or the set $\{1,2,\ldots\}$. However, no box in the leftmost column is empty.
\item Suppose there is an $a$ in box $A$ and a $b$ in box $B$. Suppose box $B$ lies to the right of box $A$ in the same row and is the leftmost such box that is nonempty. Then $a<b$ or else $a=b$ and both are unprimed numbers.
\item Suppose there is an $a$ in box $A$ and a $c$ in box $C$. Suppose box $C$ lies in the row below box $A$ and weakly to its left, and, is the rightmost such box that is nonempty. Then $a<c$ or else $a=c$ and both are primed numbers.
\end{itemize}

\end{definition}
The \emph{ left weight} of an $\mathrm{OT}$ or a $\mathrm{UT}$ is the vector whose $i^{th}$ coordinate records the number of times $i$ appears in the tableau. The \emph{right weight} of an $\mathrm{OT}$ or a $\mathrm{UT}$ is the vector whose $i^{th}$ coordinate records the number of times $i'$ appears in the tableau. The \emph{overweight} of an $\mathrm{OT}$ of shape $\lambda$ is the vector whose $i^{th}$ coordinate records the difference between the number of entries in column $i$ and the number of boxes in column $i$. The \emph{underweight} of a $\mathrm{UT}$ of shape $\lambda$ is the vector whose $i^{th}$ coordinate records the difference between the number of boxes in column $(i+1)$ and the number of entries in column $(i+1)$ (equivalently, the number of empty boxes in column $(i+1)$). (By convention, either of these differences is taken to be $0$ if the specified column is not part of the tableau.) A tableau which is both an overfull tableau and an underfull tableau (i.e., has exactly one entry per box) is called a primed tableau and the set of all such tableaux of shape $\lambda$ is denoted by $\mathrm{PT}(\lambda)$.

\begin{definition}

Let $\lambda$ be a partition and let $\mathbf{x}=(x_1,x_2,\ldots)$, $\mathbf{y}=(y_1,y_2,\ldots)$, and $\mathbf{z}=(z_1,z_2,\ldots)$ be infinite sets of indeterminants. We define polynomials:
\begin{eqnarray*}
\mathfrak{G}_{\lambda}(\mathbf{x},\mathbf{y},\mathbf{z})&=&\sum_{T \in \mathrm{OT}(\lambda)} \mathbf{x}^{\ell w(T)} \mathbf{y}^{rw(T)}\mathbf{z}^{O(T)}(-1)^{|O(T)|}\\
\mathfrak{G}^*_{\lambda}(\mathbf{x},\mathbf{y},\mathbf{z})&=&\sum_{T \in \mathrm{UT}(\lambda)} \mathbf{x}^{\ell w(T)} \mathbf{y}^{rw(T)}\mathbf{z}^{U(T)}
\end{eqnarray*}
Here, $\ell w(T)$ is the left weight of $T$, $rw(T)$ is the right weight of $T$, $O(T)$ is the overweight of $T$, and $U(T)$ is the underweight of $T$.
\end{definition}

\begin{example} An overfull and an underfull tableau are shown below.

\begin{multicols}{2}
$P=\Yboxdim{24pt}
\young({{1'11}}{{12'}}{{23'}},{{2'}}{{2}}{{3'33}},{{2'3'}}{{3}})$\\
\begin{itemize}
\item $P \in \mathrm{OT}(3,3,2)$
\item $\ell w(P)=(3,2,3)$
\item $rw(P)=(1,3,3)$
\item $O(P)=(3,1,3)$
\end{itemize}
$Q=\Yboxdim{18pt} \young({{1'}}{{}}{{1}}{{}}{{1}},{{1}}{{2'}}{{}}{{2}}{{}},{{2'}}{{}}{{3}},{{3}})$\\
\begin{itemize}
\item $Q \in \mathrm{UT}(4,4,3,1)$
\item $\ell w(Q)=(3,1,2)$
\item $rw(Q)=(1,2,0)$
\item $U(Q)=(2,1,1,1)$
\end{itemize}
\end{multicols}

\end{example}

\begin{definition}
Let $\lambda$ be a partition with conjugate, $\lambda'$. For the partition $\lambda$ the
\begin{enumerate}
\item [\emph{(\textbf{1A})}] refined symmetric Grothendieck polynomial is $\mathfrak{G}_{\lambda'}(\mathbf{0},\mathbf{x},\mathbf{z})$.
\item [\emph{(\textbf{1B})}] refined weak symmetric Grothendieck polynomial is $\mathfrak{G}_{\lambda}(\mathbf{x},\mathbf{0},\mathbf{z})$.
\item [\emph{(\textbf{2A})}] refined dual symmetric Grothendieck polynomial is $\mathfrak{G}^*_{\lambda'}(\mathbf{0},\mathbf{x},\mathbf{z})$.
\item [\emph{(\textbf{2B})}] refined dual weak symmetric Grothendieck polynomial is $\mathfrak{G}^*_{\lambda}(\mathbf{x},\mathbf{0},\mathbf{z})$.
\end{enumerate}
The nonrefined versions of these polynomials are obtained by setting $\mathbf{z}=\mathbf{1}$.
\end{definition}

\begin{remark}

These definitions coincide with the combinatorial definitions of these polynomials given elsewhere.  We now explicitly demonstrate this for the nonrefined cases using definitions that appear elsewhere in the literature word for word (up to transposition of rows and columns for consistency in some cases). 

\begin{enumerate}
\item[(\textbf{1A})] The symmetric Grothendieck polynomial associated to $\lambda$ is defined in \cite{Buch02} as $\sum \mathbf{x}^{wt(T)}$ where the sum is over all \emph{``set-valued tableaux"} of shape $\lambda$, defined in section 3 of \cite{Buch02} as:

 \emph{``If a and b are two non-empty subsets of the positive integers $\mathbb{N}$, we will write
$a < b$ if $\max(a) < \min(b)$, and $a \leq b$ if $\max(a) \leq \min(b)$. We define a set-valued
tableau to be a labeling of the boxes in a Young diagram with
finite non-empty subsets of $\mathbb{N}$, such that the rows are weakly increasing from left
to right and the columns strictly increasing from top to bottom."} The weight, $wt(T)$, of such a tableau is the vector whose $i^{th}$ coordinate records the number of times $i$ appears in the tableau.

On the other hand, by our definitions, the symmetric Grothendieck polynomial associated to $\lambda$ is $\mathfrak{G}_{\lambda'}(\mathbf{0},\mathbf{x},\mathbf{1})$, which is the generating function (weighted by right weight) over overfull tableaux of shape $\lambda'$ containing only entries from $\{1',2',\ldots\}$.  Transposing the diagram of such a tableau and removing the primes gives a ``\emph{set-valued tableau}" of shape $\lambda$ as defined above.  Moreover, this procedure sends right weight to weight and so it follows that our definition of this polynomial agrees with the cited definition.

\item[(\textbf{1B})] The weak symmetric Grothendieck polynomial associated to $\lambda$ is defined in \cite{LP07} as $\sum  \mathbf{x}^{wt(T)}$ where the sum is over all \emph{``weak set-valued tableaux"} of shape $\lambda$, defined in section 9.7 of \cite{LP07} as: 

\emph{``A weak set-valued tableau T of shape $\lambda$ is a filling of the boxes
with finite nonempty multisets of positive integers (thus, numbers in one box are not
necessarily distinct) so that
\begin{enumerate}
\item the smallest number in each box is strictly bigger than the largest number
in the box directly [above] it (if that box is present);
\item the smallest number in each box is greater than or equal to the largest
number in the box directly [to the left of] it (if that box is present)."
\end{enumerate}} \noindent The weight, $wt(T)$, of such a tableau is the vector whose $i^{th}$ coordinate records the number of times $i$ appears in the tableau.

On the other hand, by our definitions, the  weak symmetric Grothen-dieck polynomial associated to $\lambda$ is $\mathfrak{G}_{\lambda}(\mathbf{x},\mathbf{0},\mathbf{1})$, which is the generating function (weighted by left weight) over overfull tableaux of shape $\lambda$ containing only entries from $\{1,2,\ldots\}$.  But such tableaux are precisely the ``\emph{weak set-valued tableaux}" described above and their left weight is the weight of the tableau considered as a \emph{weak set-valued tableau}.  It follows that our definition of this polynomial agrees with the cited definition.

\item[(\textbf{2A})] The dual symmetric Grothendieck polynomial associated to $\lambda$ is defined in \cite{LP07} as $\sum  \mathbf{x}^{wt(T)}$ where the sum is over all \emph{``reverse plane partitions"} of shape $\lambda$, defined in section 9.1 of \cite{LP07} as: 

\emph{``A reverse plane partition $T$ of shape $\lambda$ is a filling of the boxes in $\lambda$ with positive integers so
that the numbers are weakly increasing in rows and columns."}  The weight, $wt(T)$, of such a tableau is the vector whose $i^{th}$ coordinate records the number of columns that contain an $i$.

On the other hand, by our definitions, the  dual  symmetric Grothen-dieck polynomial associated to $\lambda$ is $\mathfrak{G}^*_{\lambda'}(\mathbf{0},\mathbf{x},\mathbf{1})$, which is the generating function (weighted by right weight) over underfull tableaux of shape $\lambda'$ that only contain entries from $\{1',2',\ldots\}$.  To obtain a reverse plane partition of shape $\lambda$ from such a tableau apply the following procedure: First remove all the primes.  Then for each empty box find the closest nonempty box to its left.  Copy the entry in this box into the empty box. Transpose the result.  Since this procedure sends right weight to weight it follows that our definition of this polynomial agrees with the cited definition.

\item[(\textbf{2B})] The dual weak symmetric Grothendieck polynomial associated to $\lambda$ is defined in \cite{LP07} as $\sum \mathbf{x}^{wt(T)}$ where the sum is over all \emph{``valued-set tableaux"} of shape $\lambda$, defined in section 9.8 of \cite{LP07} as:

 \emph{``A valued-set tableaux $T$ of shape $\lambda$ is a filling of the boxes of $\lambda$ with positive integers so that
\begin{enumerate}
\item $T$ is a [usual] semistandard tableau, and
\item we are provided with the additional information of a decomposition of the
shape into a disjoint union $\lambda = \cup A_j$ of groups $A_j$ of boxes so that each
$A_j$ is connected and completely contained within a single [row] and all
boxes in each $A_j$ contain the same number."
\end{enumerate}} \noindent The weight, $wt(T)$, of such a tableau is the vector whose $i^{th}$ coordinate records the number of $A_j$ in the tableau that contain $i(s)$.

On the other hand, by our definitions, the dual weak  symmetric Grothendieck polynomial associated to $\lambda$ is $\mathfrak{G}^*_{\lambda}(\mathbf{x},\mathbf{0},\mathbf{1})$, which is the generating function (weighted by left weight) over underfull tableaux of shape $\lambda$ that only contain entries from $\{1,2,\ldots\}$.  To obtain a valued-set tableau of shape $\lambda$ from such a tableau apply the following procedure: For each nonempty box in the tableau create a group of boxes $A_j$ composed of that box along with all the empty boxes to its right (but to the left of the next nonempty box to the right).  Fill all of the boxes in each $A_j$ with the same number as appears in its leftmost box.  Since this procedure sends left weight to weight it follows that our definition of this polynomial agrees with the cited definition.  

\end{enumerate}

There are additional comparisons between our refined polynomials and those defined elsewhere in the literature to be made involving the underweight and overweight.  We briefly mention one example of this: observe the definition of ``\emph{excess}" given in relation to the definition of the refined symmetric Grothendieck polynomial in section 3 of \cite{CP19}: \emph{``Given a set-valued tableau $T$ of shape $\sigma$, define the excess of $T$, denoted $e(T)$, as the vector
$e = (e_1, e_2, \ldots)$ in which $e_i$ records the number of labels in [column] $i$ in excess of the number of boxes in [column] $i$ of $T$."} This is equivalent to our definition of overweight.

\end{remark}

\begin{definition}

Let $\mu \subseteq \lambda$ be partitions with an equal number of rows. An \emph{over flagged tableau} of shape $\lambda/\mu$, or an element of $\mathrm{OFT}(\lambda/\mu)$, is a filling of a Young diagram of shape $\lambda/\mu$ using the alphabet $1<2<\cdots$ such that:
\begin{itemize}
\item Each box in row $i$ of $\lambda/\mu$ contains one element from $\{1,2,\ldots,\mu_i\}$.
\item Suppose box $A$ lies immediately to the left of box $B$. Suppose there is an $a$ in $A$ and a $b$ in $B$. Then $a\geq b$.
\item Suppose box $A$ lies immediately to the above of box $C$. Suppose there is an $a$ in $A$ and a $c$ in $C$. Then $a>c$.
\end{itemize}
The weight, $wt(P)$, of an $\mathrm{OFT}$, $P$, is the vector whose $i^{th}$ coordinate records the number of times $i$ appears in the tableau.

Let $\mu \subseteq \lambda$ be partitions with an equal number of rows. An \emph{under flagged tableau} of shape $\lambda/\mu$, or an element of $\mathrm{UFT}(\lambda/\mu)$ is a filling of a Young diagram of shape $\lambda/\mu$ using the alphabet $1' <2'<\cdots$ such that:
\begin{itemize}
\item Each box in row $i$ of $\lambda/\mu$ contains one element from $\{1',\ldots,(\lambda_i-1)'\}$.
\item Suppose box $A$ lies immediately to the left of box $B$. Suppose there is an $a$ in $A$ and a $b$ in $B$. Then $a< b$.
\item Suppose box $A$ lies immediately to the above of box $C$. Suppose there is an $a$ in $A$ and a $c$ in $C$. Then $a\leq c$.
\end{itemize}
The weight, $wt(P)$, of a $\mathrm{UFT}$, $P$ is the vector whose $i^{th}$ coordinate records the number of times $i'$ appears in the tableau.
\end{definition}

\begin{example} An over flagged and an under flagged tableau are shown below.
\begin{multicols}{2}
$P=\Yboxdim{18pt}
\young(****42,***321,**221,*111)$\\
\begin{itemize}
\setlength{\itemindent}{-1em}
\item \small{$P \in \mathrm{OFT}((6,6,5,4)/(4,3,2,1))$}
\item $wt(P)=(5,4,1,1)$
\end{itemize}
$Q=\Yboxdim{18pt}
\young(****{{1'}}{{5'}},***{{2'}}{{3'}}{{5'}},**{{1'}}{{2'}}{{4'}},*{{1'}}{{2'}}{{3'}})$\\
\begin{itemize}
\setlength{\itemindent}{-2em}
\item $Q \in \mathrm{UFT}((6,6,5,4)/(4,3,2,1))$
\item $wt(Q)=(3,3,2,1,2)$
\end{itemize}
\end{multicols}

\end{example}
\begin{convention}\label{difrows} If $\lambda \supseteq \mu$ are partitions with different numbers of rows, then we set $\mathrm{OFT}(\lambda/\mu)=\emptyset=\mathrm{UFT}(\lambda/\mu)$.
\end{convention}

\begin{lemma}\label{rskjdt} There are bijections between the following sets:
\begin{enumerate}
\item $\mathrm{OT}(\mu)$ and $\{(P,Q): \exists \lambda \supseteq \mu: P \in \mathrm{PT}(\lambda),Q \in \mathrm{OFT}(\lambda/\mu)\}$.
\item $\mathrm{UT}(\lambda)$ and $\{(P,Q): \exists \mu \subseteq \lambda: P \in \mathrm{PT}(\mu),Q \in \mathrm{UFT}(\lambda/\mu)\}$.
\end{enumerate}
In case (1), if $T \rightarrow (P,Q)$ then $\ell w(T)=\ell w(P)$, $rw(T)=rw(P)$, and $O(T)=wt(Q)$.\\
In case (2), if $T \rightarrow (P,Q)$ then $\ell w(T)=\ell w(P)$, $rw(T)=rw(P)$, and $U(T)=wt(Q)$.
\end{lemma}

We will be using certain versions of two well known combinatorial algorithms in the following proof, both adapted to the case of primed tableaux and assuming the order $1'<1<2'<2< \cdots$.  The column RSK insertion algorithm for primed tableaux is defined as follows:  Given $T \in \mathrm{PT}(\lambda)$ and $a \in \{1',1,2',2,\ldots\}$ we insert $a$ into $T$ by first inserting it into the leftmost column of $T$, where $a$ replaces the smallest entry greater than $a$ if $a$ is primed and the smallest entry greater than or equal to $a$ if $a$ is unprimed.  If no such entry exists, $a$ is appended to the bottom of the column.  The replaced entry is then inserted in the next column to the right and the procedure continues until an entry is appended to some column.  

The jeu de taquin algorithm for primed tableaux is defined as:  Given $T \in \mathrm{PT}(\lambda)$ remove the entry from the top left box, $b$, of $T$ (in the proof below the box will have already been removed before we start).  From the box below $b$ and the box to the right of $b$ select the smaller entry (if the entries are equal, select from the box below if the entries are unprimed, and, select from the box to the right  if the entries are primed).  Move this entry into $b$.  A new box is now empty.  Repeat the procedure with the new empty box taking the role of $b$ and continue like this until the empty box is an outer corner box of $\lambda$.

\begin{proof}
We construct each bijection. We will use RSK column insertion in the first and jeu de taquin in the second. 
\begin{enumerate}
\item Start with $T \in \mathrm{OT}(\mu)$ and construct $(P,Q)$ as follows. Beginning with the rightmost column, that is column $\mu_1$, and working to the leftmost column, that is, column $1$, do as follows. From each box in the current column, say column $i$, remove all but the smallest entry. Now, in weakly decreasing order, insert the removed entries to the tableau formed by the columns $i+1,i+2,\ldots$ using RSK column insertion. Let $P$ be the resulting tableau and suppose it has shape $\lambda$. Now construct a tableau $Q$ of shape $\lambda/\mu$ by placing an $i$ in each box that corresponds to the position of a box appended during the RSK insertions that occurred after removing entries from column $i$.
\item Start with $T \in \mathrm{UT}(\lambda)$ and construct $(P,Q)$ as follows. Beginning with the rightmost column, that is column $\lambda_1$, and working to the second to leftmost column, that is, column $2$, do as follows. Starting with the lowest and working to the highest, do the following for each empty box, $b$, in column $i$. Consider the tableau, $R$, whose upper left corner box is $b$ (that is, the tableau composed of all boxes lying weakly below and weakly to the right of $b$). $R$ has exactly one empty box which is box $b$. Apply jdt into box $b$. This results  in an empty corner box appearing in $R$.  Remove this box.  After this procedure has been done for each box $b$ for each column $i$ from $\lambda_1$ to $2$ define the result to be $P$ and denote its shape by $\mu$. Now construct a tableau $Q$ of shape $\lambda/\mu$ by placing a $(i-1)'$ in each box that corresponds to the position of a box removed after jdt into a box in column $i$.
\end{enumerate}
The $P$-tableaux constructed are valid primed tableaux by the properties of RSK and jdt. The $Q$-tableaux constructed are valid column strict and row strict tableaux respectively, also by properties of RSK and jdt. On the other hand, the flag conditions for an $\mathrm{OT}$ or a $\mathrm{UT}$ are ensured directly from the construction. Moreover, the weight conditions $\ell w(T)=\ell w(P)$, $rw(T)=rw(P)$ are immediate since RSK and jdt don't affect the set of entries appearing in a tableau, and the weight conditions $O(T)=wt(Q)$ and $U(T)=wt(Q)$ follow by construction. Finally, it is not difficult to see how to invert the maps using the fact that RSK and jdt are themselves invertible.
\end{proof}

\begin{example}
An example of bijection $(1)$ of Lemma \ref{rskjdt}. In each step the newly colored entry is to be RSK inserted into the column to its right. The final colored tableau is $P$ and $Q$ is shown directly below it.
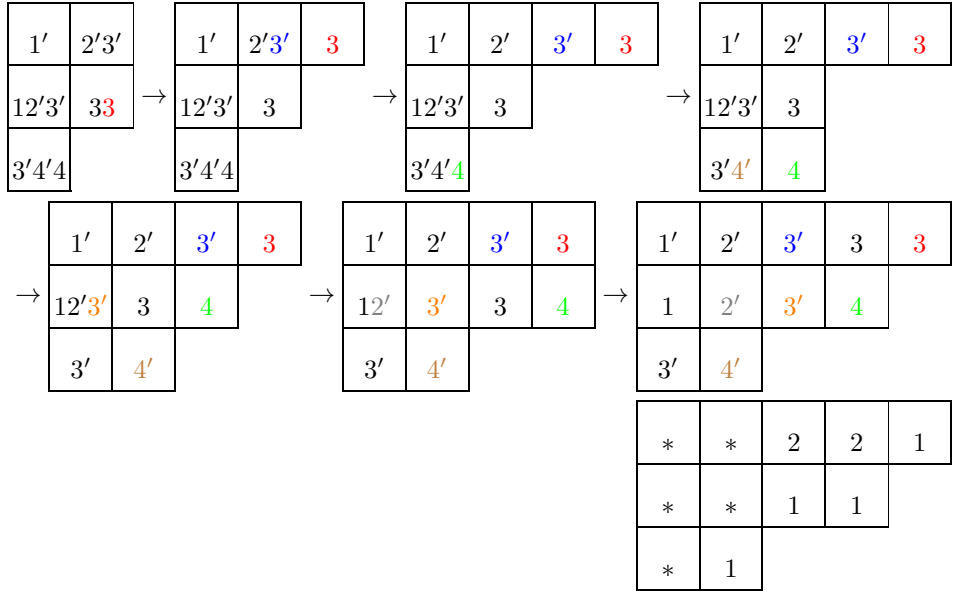
\begin{eqnarray*}
\Yboxdim{24.1pt}
\young({{1'}}{{2'3'}},{{12'3'}}{{3\textcolor{red}{3}}},{{3'4'4}}) \rightarrow
\young({{1'}}{{2'\textcolor{blue}{3'}}}{{\textcolor{red}{3}}},{{12'3'}}{{3}},{{3'4'4}}) \rightarrow
\young({{1'}}{{2'}}{{\textcolor{blue}{3'}}}{{\textcolor{red}{3}}},{{12'3'}}{{3}},{{3'4'\textcolor{green}{4}}}) \rightarrow
\young({{1'}}{{2'}}{{\textcolor{blue}{3'}}}{{\textcolor{red}{3}}},{{12'3'}}{{3}},{{3'\textcolor{brown}{4'}}}{{\textcolor{green}{4}}})\\
\Yboxdim{24.1pt}
\rightarrow \young({{1'}}{{2'}}{{\textcolor{blue}{3'}}}{{\textcolor{red}{3}}},{{12'\textcolor{orange}{3'}}}{{3}}{{\textcolor{green}{4}}},{{3'}}{{\textcolor{brown}{4'}}})
\rightarrow \young({{1'}}{{2'}}{{\textcolor{blue}{3'}}}{{\textcolor{red}{3}}},{{1\textcolor{gray}{2'}}}{{\textcolor{orange}{3'}}}{{3}}{{\textcolor{green}{4}}},{{3'}}{{\textcolor{brown}{4'}}})
\rightarrow \young({{1'}}{{2'}}{{\textcolor{blue}{3'}}}{{3}}{{\textcolor{red}{3}}},{{1}}{{\textcolor{gray}{2'}}}{{\textcolor{orange}{3'}}}{{\textcolor{green}{4}}},{{3'}}{{\textcolor{brown}{4'}}})\\
\Yboxdim{24.1pt} \young(**221,**11,*1)
\end{eqnarray*}
\end{example}

\begin{example}
An example of bijection $(2)$ of Lemma \ref{rskjdt}. In each step the jeu de taquin algorithm is to be applied into the box marked by the $\textcolor{red}{\times}$. The final tableau is $P$ and $Q$ is shown directly below it.

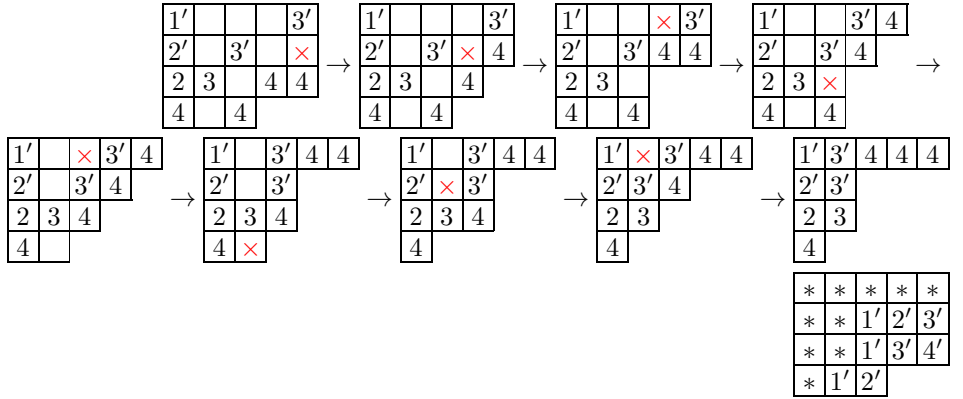
\begin{eqnarray*}
\Yboxdim{12pt}
\young({{1'}}{{}}{{}}{{}}{{3'}},{{2'}}{{}}{{3'}}{{}}{{\textcolor{red}{\times}}},{{2}}{{3}}{{}}{{4}}{{4}},{{4}}{{}}{{4}})\rightarrow
\young({{1'}}{{}}{{}}{{}}{{3'}},{{2'}}{{}}{{3'}}{{\textcolor{red}{\times}}}{{4}},{{2}}{{3}}{{}}{{4}},{{4}}{{}}{{4}})\rightarrow
\young({{1'}}{{}}{{}}{{\textcolor{red}{\times}}}{{3'}},{{2'}}{{}}{{3'}}{{4}}{{4}},{{2}}{{3}}{{}},{{4}}{{}}{{4}})\rightarrow
\young({{1'}}{{}}{{}}{{3'}}{{4}},{{2'}}{{}}{{3'}}{{4}},{{2}}{{3}}{{\textcolor{red}{\times}}},{{4}}{{}}{{4}})\rightarrow\\
\Yboxdim{12pt}
\young({{1'}}{{}}{{\textcolor{red}{\times}}}{{3'}}{{4}},{{2'}}{{}}{{3'}}{{4}},{{2}}{{3}}{{4}},{{4}}{{}})\rightarrow
\young({{1'}}{{}}{{3'}}{{4}}{{4}},{{2'}}{{}}{{3'}},{{2}}{{3}}{{4}},{{4}}{{\textcolor{red}{\times}}})\rightarrow
\young({{1'}}{{}}{{3'}}{{4}}{{4}},{{2'}}{{\textcolor{red}{\times}}}{{3'}},{{2}}{{3}}{{4}},{{4}})\rightarrow
\young({{1'}}{{\textcolor{red}{\times}}}{{3'}}{{4}}{{4}},{{2'}}{{3'}}{{4}},{{2}}{{3}},{{4}}) \rightarrow
\young({{1'}}{{3'}}{{4}}{{4}}{{4}},{{2'}}{{3'}},{{2}}{{3}},{{4}})\\
\Yboxdim{12pt}
\young(*****,**{{1'}}{{2'}}{{3'}},**{{1'}}{{3'}}{{4'}},*{{1'}}{{2'}})
\end{eqnarray*}
\end{example}

The following two main facts are the main results of this paper. These facts along with the results so far stated give a complete combinatorial understanding of two types of duality appearing in Grothendieck polynomials. These facts will be proven in the next section.

\begin{mainfact}\label{mf1}
Let $\omega_{\mathbf{x}}$ denote the involution on functions symmetric in $\mathbf{x}$ over the ring $\mathbb{Z}[\mathbf{y},\mathbf{z}]$ defined by $\omega_{\mathbf{x}}(s_{\lambda}(\mathbf{x}))=s_{\lambda'}(\mathbf{x})$. Similarly, let $\omega_{\mathbf{y}}$ denote the involution on functions symmetric in $\mathbf{y}$ over the ring $\mathbb{Z}[\mathbf{x},\mathbf{z}]$ defined by $\omega_{\mathbf{y}}(s_{\lambda}(\mathbf{y}))=s_{\lambda'}(\mathbf{y})$. Then we have that:
\begin{eqnarray*}
\omega_{\mathbf{x}}\omega_{\mathbf{y}}\left(\sum_{P \in \mathrm{PT}(\lambda)} \mathbf{x}^{\ell w(P)} \mathbf{y}^{rw(P)}\right)=\sum_{P \in \mathrm{PT}(\lambda)} \mathbf{y}^{\ell w(P)} \mathbf{x}^{rw(P)}
\end{eqnarray*}
\end{mainfact}

\begin{corollary}\label{duo1} We have:
\begin{align*}
\omega_{\mathbf{x}}\omega_{\mathbf{y}}\mathfrak{G}_{\mu}(\mathbf{x},\mathbf{y},\mathbf{z})=\mathfrak{G}_{\mu}(\mathbf{y},\mathbf{x},\mathbf{z})\\
\omega_{\mathbf{x}}\omega_{\mathbf{y}}\mathfrak{G}^*_{\lambda}(\mathbf{x},\mathbf{y},\mathbf{z})=\mathfrak{G}^*_{\lambda}(\mathbf{y},\mathbf{x},\mathbf{z})\\
\end{align*}
\end{corollary}

\begin{proof}
Using Lemma \ref{rskjdt} we may write:
\begin{align}
&\mathfrak{G}_{\mu}(\mathbf{x},\mathbf{y},\mathbf{z})=\sum_{\lambda \supseteq \mu}& (-1)^{|\lambda|-|\mu|}&\sum_{Q \in \mathrm{OFT}(\lambda/\mu)}\mathbf{z}^{O(Q)}\sum_{P \in \mathrm{PT}(\lambda)} \mathbf{x}^{\ell w(P)} \mathbf{y}^{rw(P)}\label{pf1}&\\
&\mathfrak{G}^*_{\lambda}(\mathbf{x},\mathbf{y},\mathbf{z})=\sum_{\mu \subseteq \lambda} & & \sum_{Q \in \mathrm{UFT}(\lambda/\mu)}\mathbf{z}^{U(Q)}\sum_{P \in \mathrm{PT}(\lambda)} \mathbf{x}^{\ell w(P)} \mathbf{y}^{rw(P)}& \label{pf2}
\end{align}
The corollary now follows from main fact \ref{mf1} and the linearity of $\omega_{\mathbf{x}}$ and $\omega_{\mathbf{y}}$.
\end{proof}

\begin{remark}
Setting $\mathbf{y}=\mathbf{0}$ in the first equation of corollary \ref{duo1} we see that the refined weak symmetric Grothendieck polynomial maps under $\omega_{\mathbf{x}}$ to the refined symmetric Grothendieck polynomial of conjugate shape. Setting $\mathbf{y}=\mathbf{0}$ in the second equation of corollary \ref{duo1} we see that the refined dual weak symmetric Grothendieck polynomial maps under $\omega_{\mathbf{x}}$ to the refined dual symmetric Grothendieck polynomial of conjugate shape. In other words this shows $\textbf{1A} \xrightarrow{\omega} \textbf{1B}$ and $\textbf{2A} \xrightarrow{\omega} \textbf{2B}$.
\end{remark}

\begin{mainfact}\label{mf2}
Let $\mu \subset \lambda$ be partitions with the same number of rows. Then:
\begin{eqnarray*}
\sum_{\mu \subseteq \rho \subseteq \lambda} \sum_{(P,Q) \in T_{\rho}}(-1)^{|\rho|-|\mu|}\mathbf{z}^{wt(P)}\mathbf{z}^{wt(Q)}=0
\end{eqnarray*}
where $T_{\rho}$ is the set of all pairs $(P,Q)$ with $P \in \mathrm{OFT}(\rho/\mu)$ and $Q \in \mathrm{UFT}(\lambda/\rho)$.
\end{mainfact}

\begin{corollary}\label{duo2}
Let $\langle,\rangle$ denote the bilinear form on functions symmetric in $\mathbf{x}$ over the ring $\mathbb{Z}[\mathbf{z}]$ defined by $\langle s_{\mu}(\mathbf{x}),s_{\lambda}(\mathbf{x})\rangle=\delta_{\mu,\lambda}$ and extended $\mathbb{Z}[\mathbf{z}]$-bilinearly. Then we have:
\begin{eqnarray*}
\langle \mathfrak{G}_{\mu}(\mathbf{x},\mathbf{0},\mathbf{z}), \mathfrak{G}^*_{\lambda}(\mathbf{x},\mathbf{0},\mathbf{z}) \rangle = \delta_{\mu,\lambda}\\
\langle \mathfrak{G}_{\mu}(\mathbf{0},\mathbf{x},\mathbf{z}), \mathfrak{G}^*_{\lambda}(\mathbf{0},\mathbf{x},\mathbf{z}) \rangle = \delta_{\mu,\lambda}\\
\end{eqnarray*}

\end{corollary}

\begin{proof}

Since a primed tableau with right weight (resp. left weight) of $\mathbf{0}$ is just a semistandard Young tableau (resp. conjugated semistandard Young tableau) it follows from the formulas \ref{pf1} and \ref{pf2} appearing in the proof of Corollary \ref{duo1} that $\langle \mathfrak{G}_{\mu}(\mathbf{x},\mathbf{0},\mathbf{z}), \mathfrak{G}^*_{\lambda}(\mathbf{x},\mathbf{0},\mathbf{z}) \rangle$ is equal to:
\begin{align*}
\langle \left( \sum_{\rho \supseteq \mu} (-1)^{|\rho|-|\mu|}\sum_{P \in \mathrm{OFT}(\rho/\mu)}\mathbf{z}^{O(P)}s_{\rho}(\mathbf{x})\right) , \left(\sum_{\sigma \subseteq \lambda} \sum_{Q \in \mathrm{UFT}(\lambda/\sigma)}\mathbf{z}^{U(Q)}s_{\sigma}(\mathbf{x})\right)\rangle\\
\end{align*}
and that $\langle \mathfrak{G}_{\mu}(\mathbf{0},\mathbf{x},\mathbf{z}), \mathfrak{G}^*_{\lambda}(\mathbf{0},\mathbf{x},\mathbf{z}) \rangle$ is equal to:
\begin{align*}
\langle \left( \sum_{\rho \supseteq \mu} (-1)^{|\rho|-|\mu|}\sum_{P \in \mathrm{OFT}(\rho/\mu)}\mathbf{z}^{O(P)}s_{\rho'}(\mathbf{x})\right) , \left(\sum_{\sigma \subseteq \lambda} \sum_{Q \in \mathrm{UFT}(\lambda/\sigma)}\mathbf{z}^{U(Q)}s_{\sigma'}(\mathbf{x})\right)\rangle\\
\end{align*}
If $\mu=\lambda$ then it is clear there is exactly one nonzero term in the expansion of these inner products which occurs when $\mu=\rho=\sigma$ and that this term has coefficient $(-1)^0 \mathbf{z}^{\mathbf{0}}=1$. If $\mu \not\subseteq \lambda$ then it is clear that all terms in the expansion of the inner products must be $0$.  If $\mu \subset \lambda$ have different numbers of rows then it follows from convention \ref{difrows} that all terms in the expansion of the inner products must be $0$. Finally, if $\mu \subset \lambda$ have the same number of rows then it follows from main fact \ref{mf2} that the inner products evaluate to $0$.
\end{proof}

\begin{remark}
Corollary \ref{duo2} says that under $\langle,\rangle$ the refined weak symmetric Grothendieck polynomial is dual to the refined dual weak symmetric Grothen-dieck polynomial and that the refined symmetric Grothendieck polynomial is dual to the refined dual symmetric Grothendieck polynomial. I.e., $\textbf{1A} \xrightarrow{\langle,\rangle} \textbf{2A}$ and $\textbf{1B} \xrightarrow{\langle,\rangle} \textbf{2B}$.
\end{remark}

\section{Proof of the Main Facts}

\begin{definition}
Let $\prec$ denote any total order on the set $\{1',1,2',2',\ldots\}$. We define a primed tableau with respect to $\prec$ of shape $\lambda/\mu$ to be a filling of the shape $\lambda/\mu$ such that:
\begin{itemize}
\item Each box of $\lambda/\mu$ contains exactly one of $\{1',1,2',2',\ldots\}$.
\item The rows of $\lambda/\mu$ weakly increase under $\prec$ left to right.
\item The columns of $\lambda/\mu$ weakly increase under $\prec$ top to bottom.
\item There is at most one $i$ in each column for each $i$.
\item There is at most one $i'$ in each row for each $i$.
\end{itemize}
The set of such tableaux is denoted by $\mathrm{PT}_{\prec}(\lambda/\mu)$. The left weight of a $\mathrm{PT}_{\prec}$ is the vector whose $i^{th}$ coordinate records the number of times $i$ appears in the tableau. The right weight of a $\mathrm{PT}_{\prec}$ is the vector whose $i^{th}$ coordinate records the number of times $i'$ appears in the tableau.
\end{definition}
Note that if $\prec$ is the order $1'\prec1\prec2'\prec2\prec \cdots$ then $\mathrm{PT}_{\prec}(\lambda/\emptyset)=\mathrm{PT}(\lambda)$.

\begin{lemma}\label{ordering}
The total order chosen in the definition above is irrelevant. In other words, given any two total orderings $\prec$ and $\vartriangleleft$ there is a left weight and right weight preserving bijection between $\mathrm{PT}_{\prec}(\lambda/\mu)$ and $\mathrm{PT}_{\vartriangleleft}(\lambda/\mu)$.
\end{lemma}

\begin{proof}
First we prove the lemma in the case that there are some $i$ and $j$ (possibly equal) such that $i \prec j'$ and $j'\vartriangleleft i$ and all other pairs of letters have the same relationship in both orders. From this assumption it follows that under both $\prec$ and $\vartriangleleft$ there is no other letter between $i$ and $j'$.

If $T \in \mathrm{PT}_{\prec}(\lambda/\mu)$ or $T \in \mathrm{PT}_{\vartriangleleft}(\lambda/\mu)$, let $\mathcal{B}(T)$ denote all of the boxes of $\lambda/\mu$ that contain an $i$ or a $j'$. If $\lambda=(\lambda_1,\ldots,\lambda_{\ell})$ then label the the upper left to lower right diagonals of $\lambda$ by $\{d_{-\ell+1}, \ldots,d_{-1},d_0,d_1,\ldots, d_{\lambda_1-1}\}$. It is clear that either $\mathcal{B}(T) \cap d_s$ is a single box or is empty for each $s$. If $p$ is minimal and $q$ is maximal such that the adjacent diagonals $d_p,d_{p+1},\ldots,d_{q-1},d_q$ each have nonempty intersection with $\mathcal{B}(T)$ then we call $B=\{d_p \cap \mathcal{B}(T) , \ldots, d_q \cap \mathcal{B}(T)\}$ a connected component of $\mathcal{B}(T)$.  (Note that such a $B$ is in fact a ribbon of $T$).

Define a map $\nearrow$ from $\mathrm{PT}_{\prec}(\lambda/\mu)$ to $\mathrm{PT}_{\vartriangleleft}(\lambda/\mu)$ as follows. Suppose that $T \in \mathrm{PT}_{\prec}(\lambda/\mu)$. Perform the following to each connected component $B$ of $\mathcal{B}(T)$.
\begin{itemize}
\item Remove the entry that appears in the upper rightmost box of $B$. Record what you have removed.
\item Move every remaining $i$ one box to the right and every remaining $j'$ one box up.
\item Fill the lower leftmost box of $B$ with the entry recorded in step one.
\end{itemize}

Next define a map $\swarrow$ from $\mathrm{PT}_{\vartriangleleft}(\lambda/\mu)$ to $\mathrm{PT}_{\prec}(\lambda/\mu)$ as follows. Suppose that $T \in \mathrm{PT}_{\vartriangleleft}(\lambda)$. Perform the following to each connected component $B$ of $\mathcal{B}(T)$.
\begin{itemize}
\item Remove the entry that appears in the lower leftmost box of $B$. Record what you have removed.
\item Move every remaining $i$ one box to the left and every remaining $j'$ one box down.
\item Fill the upper rightmost box of $B$ with the entry recorded in step one.
\end{itemize}

It is not difficult to check that the maps $\nearrow$ and $\swarrow$ are well-defined and preserve weights and are mutual inverses.  This proves the lemma in the case that there are some $i$ and $j$ (possibly equal) such that $i \prec j'$ and $j'\vartriangleleft i$ and all other pairs of letters have the same relationship in both orders.  

Now suppose we are given two arbitrary orderings $\prec$ and $\vartriangleleft$. By what was just proved we may successively alter $\prec$ until all unprimed entries precede all primed entries and we may do so similarly for $\vartriangleleft$.  Thus we may assume that $i \prec j'$  and $i \vartriangleleft j'$  for all $i$ and $j$.  But in this case a primed tableau may be thought of as a pair of tableaux composed of the tableau formed by the unprimed entries and the tableau formed by the primed entries. Thus we may write:
\begin{eqnarray*}
\mathrm{PT}_{\prec}(\lambda/\mu)= \bigcup_{\mu \subseteq \rho \subseteq \lambda} \mathrm{SSYT}_{\prec}(\rho/\mu) \times \mathrm{SSYT}'_{\prec}(\lambda/\rho)
\end{eqnarray*}
where $\mathrm{SSYT}_{\prec}(\rho/\mu)$ is the set of skew semistandard Young tableaux of shape $(\rho/\mu)$in the alphabet $\{1,2,\cdots\}$ under the order $\prec$ as restricted to $\{1,2,\cdots\}$ and $\mathrm{SSYT}'_{\prec}(\lambda/\rho)$ is the set of conjugate skew semistandard Young tableaux of shape $(\lambda/\rho)$ in the alphabet $\{1',2',\cdots\}$ under the order $\prec$ as restricted to $\{1',2',\cdots\}$. But all skew Schur functions are symmetric so this implies that $\prec$ may be replaced with $\vartriangleleft$ in the right hand side of the equation above, which makes it equal to $\mathrm{PT}_{\vartriangleleft}(\lambda/\mu)$ by the same logic with which we arrived at this equation. 
\end{proof}

\begin{example}
Suppose that we are given the following orderings:
\begin{eqnarray*}
1' \prec 1 \prec 2 \prec 3 \prec 2' \prec 4 \prec 3' \prec 4'\\
1' \vartriangleleft 1 \vartriangleleft 2 \vartriangleleft 2' \vartriangleleft 3 \vartriangleleft 4 \vartriangleleft 3' \vartriangleleft 4'
\end{eqnarray*}
Then under the map $\nearrow$ with $i=3$ and $j=2$ the tableau $S \in \mathrm{PT}_{\prec}(\lambda/\mu)$ on the left below is sent to the tableau $T \in \mathrm{PT}_{\vartriangleleft}(\lambda/\mu)$ on the right below. Note that $\mathcal{B}(S)$ is composed of $3$ connected components of sizes $1$, $7$, and $3$.
\begin{eqnarray*}
S=\young(::{{1'}}122{{\textcolor{red}{3}}}{{\textcolor{red}{3}}}{{\textcolor{red}{2'}}},{{1'}}11{{\textcolor{red}{3}}}{{\textcolor{red}{3}}}{{\textcolor{red}{3}}}{{3'}}{{4'}},{{1'}}22{{\textcolor{red}{2'}}}44{{3'}},1{{\textcolor{red}{3}}}{{\textcolor{red}{3}}}{{\textcolor{red}{2'}}}{{3'}},24{{3'}},{{\textcolor{red}{3}}}{{3'}})
\nearrow \,\,\,\,\,\,\,\,\,\,\,\,
\young(::{{1'}}122{{\textcolor{red}{2'}}}{{\textcolor{red}{3}}}{{\textcolor{red}{3}}},{{1'}}11{{\textcolor{red}{2'}}}{{\textcolor{red}{3}}}{{\textcolor{red}{3}}}{{3'}}{{4'}},{{1'}}22{{\textcolor{red}{2'}}}44{{3'}},1{{\textcolor{red}{3}}}{{\textcolor{red}{3}}}{{\textcolor{red}{3}}}{{3'}},24{{3'}},{{\textcolor{red}{3}}}{{3'}})=T
\end{eqnarray*}
\end{example}

\begin{proof}[Proof of Main Fact \ref{mf1}]
Consider the orders:
\begin{eqnarray*}
1'\prec 1 \prec 2' \prec 2 \prec \cdots\\
1\vartriangleleft 2 \vartriangleleft \cdots 1' \vartriangleleft 2' \vartriangleleft \cdots\\
1' \blacktriangleleft 2' \blacktriangleleft \cdots 1 \blacktriangleleft 2 \blacktriangleleft \cdots
\end{eqnarray*}

Since an element of $ \mathrm{PT}_{\vartriangleleft}(\lambda/\emptyset)$ can be thought as a semistandard Young tableau of some shape $\mu \subseteq \lambda$ along with a conjugate semistandard Young tableau (with its entries primed) of shape $\lambda/\mu$, we see that:
\begin{eqnarray}
\sum_{P \in \mathrm{PT}_{\vartriangleleft}(\lambda/\emptyset)} \mathbf{x}^{\ell w(P)} \mathbf{y}^{rw(P)}=\sum_{\mu \subseteq \lambda} s_{\mu}(\mathbf{x}) s_{\lambda'/\mu'}(\mathbf{y}) \label{before}
\end{eqnarray}

On the other hand, an element $\mathrm{PT}_{\blacktriangleleft}(\lambda/\emptyset)$ can be thought as a conjugate semistandard Young tableau (with its entries primed) of some shape $\mu \subseteq \lambda$ along with a semistandard Young tableau of shape $\lambda/\mu$, so that:
\begin{eqnarray}
\sum_{P \in \mathrm{PT}_{\blacktriangleleft}(\lambda/\emptyset)} \mathbf{y}^{\ell w(P)} \mathbf{x}^{rw(P)}=\sum_{\mu \subseteq \lambda} s_{\mu'}(\mathbf{x}) s_{\lambda/\mu}(\mathbf{y}) \label{above}
\end{eqnarray}

Since the right hand side of the equation \ref{above} is obtained from the right hand side of  equation \ref{before} by applying $\omega_{\mathbf{x}}\omega_{\mathbf{y}}$ the same is true of the left hand sides. Moreover by Lemma \ref{ordering} both $\mathrm{PT}_{\vartriangleleft}(\lambda/\emptyset)$ and $\mathrm{PT}_{\blacktriangleleft}(\lambda/\emptyset)$ may be replaced by $\mathrm{PT}_{\prec}(\lambda/\emptyset)$ which is turn equivalent to $\mathrm{PT}(\lambda)$. This proves that:
\begin{eqnarray*}
\omega_{\mathbf{x}}\omega_{\mathbf{y}}\left(\sum_{P \in \mathrm{PT}(\lambda)} \mathbf{x}^{\ell w(P)} \mathbf{y}^{rw(P)}\right)=\sum_{P \in \mathrm{PT}(\lambda)} \mathbf{y}^{\ell w(P)} \mathbf{x}^{rw(P)}
\end{eqnarray*}
\end{proof}

\begin{definition}
Fix the order $\cdots \prec 2 \prec 1 \prec 1' \prec 2' \prec \cdots$. Define a primed flagged tableau of shape $\lambda/\mu$ to be a tableau, $P$, that is an element of $\mathrm{PT}_{\prec}(\lambda/\mu)$ such that:
\begin{itemize}
\item If row $i$ of $P$ contains a $j$ then $j \leq \mu_i$
\item If row $i$ of $P$ contains a $j'$ then $j < \lambda_i$.
\end{itemize}
Let $\mathrm{PFT}(\lambda/\mu)$ denote the set of all primed flagged tableaux of shape $\lambda/\mu$. We say a ``row is $\prec$" if the entries are weakly increasing left ro right under $\prec$ and there are no repeated primed entries. We say a ``column is $\prec$" if the entries are weakly increasing top to bottom under $\prec$ and there are no repeated unprimed entries. The condition concerning the nonprimed entries of row $i$ is called the nonprimed row $i$ flag condition. The condition concerning the primed entries of row $i$ is called the primed row $i$ flag condition. In other words, $P \in \mathrm{PFT}(\lambda/\mu)$ if and only all its rows and columns are $\prec$ and it satisfies the nonprimed and primed row $i$ flag conditions for each $i$.

\end{definition}

\begin{lemma}\label{z}
Let $\mu \subset \lambda$ be partitions with the same number of rows.
\begin{eqnarray*}
\sum_{P \in \mathrm{PFT}_{\prec}(\lambda/\mu)} (-1)^{\#(P)}\mathbf{z}^{\ell w(P)}\mathbf{z}^{rw(P)}=0
\end{eqnarray*}
Where $\#(P)$ is the number of nonprimed entries in $P$, $\ell w(P)$ is the left weight of $P$, and $rw(P)$ is the right weight of $P$.
\end{lemma}
\begin{proof}
We complete the proof by constructing a (fixed point free) sign reversing involution on $ \mathrm{PFT}(\lambda/\mu)$ which preserves the value of $\ell w(P)+rw(P)$ and where the sign is the value of $(-1)^{\#(P)}$. The involution, $\iota$, is described as follows:

Let $P \in \mathrm{PFT}(\lambda/\mu)$. Let $m$ be the smallest integer such that $P$ contains an $m$ or an $m'$. Now, let $i$ be minimal such that row $i$ of $P$ contains an $m$ or $m'$. Finally, let $j$ be maximal such that the box, $b$, of $P$, in row $i$ and column $j$ contains an $m$ or an $m'$. Define $\iota(P)$ to be the tableau obtained by:
\begin{itemize}
\item If $b$ contains an $m$, replace it with an $m'$.
\item If $b$ contains an $m'$, replace it with an $m$.
\end{itemize}

First suppose that box $b$ of $P$ contains an $m$. If $\lambda/\mu$ has a box to the right of $b$ then $P$ must have in this box an $s$ or an $s' $ for some $s>m$ by the minimality of $m$ and the maximality of $j$. But it cannot be an $s$ because row $i$ of $P$ is assumed to be $\prec$. Since $m' \prec s'$ it follows that row $i$ of $\iota(P)$ is $\prec$.
If $\lambda/\mu$ has a box below $b$ then $P$ must have in this box a $t$ or a $t' $ for some $t \geq m$ by the minimality of $m$. But it cannot be a $t$ because column $j$ of $P$ is assumed to be $\prec$. Since $m' \preceq t'$ it follows that column $j$ of $\iota(P)$ is $\prec$. Consequently, all rows and columns of $\iota(P)$ are $\prec$.
Next, we have $m \leq \mu_i$ because $P$ is assumed to satisfy the nonprimed row $i$ flag condition. But $\mu_i<\lambda_i$ since we know row $i$ of $\lambda/\mu$ has at least one box. Thus $m < \lambda_i$ and it follows that $\iota(P)$ satisfies the primed row $i$ flag condition. Consequently, $\iota(P)$ satisfies all flag conditions. This shows that $\iota(P) \in \mathrm{PFT}(\lambda/\mu)$.

Now suppose that box $b$ of $P$ contains an $m'$. If $\lambda/\mu$ has a box above $b$ then $P$ must have in this box an $s$ or an $s' $ for some $s>m$ by the minimality of $m$ and the minimality of $i$. But it cannot be an $s'$ because column $j$ of $P$ is assumed to be $\prec$. Since $s \prec m$ it follows that column $j$ of $\iota(P)$ is $\prec$.
If $\lambda/\mu$ has a box to the left of $b$ then $P$ must have in this box a $t$ or a $t' $ for some $t \geq m$ by the minimality of $m$. But it cannot be a $t'$ because row $i$ of $P$ is assumed to be $\prec$. Since $t \preceq m$ it follows that row $i$ of $\iota(P)$ is $\prec$. Consequently, all rows and columns of $\iota(P)$ are $\prec$.
Next, suppose that $m>\mu_i$. As stated, if $\lambda/\mu$ has a box to the left of $b$ then $P$ has in this box some $t$ with $t \geq m$. But this would contradict the nonprimed row $i$ flag condition so $b$ must be the leftmost box in row $i$ of $\lambda/\mu$. It follows from this and the fact that row $i$ of $P$ is $\prec$ that the box in row $i$ and column $\lambda_i$ of $P$ contains $n'$ for some $n \geq \lambda_i-\mu_i-1+m$. But this means that $n \geq \lambda_i$ since we assumed $m>\mu_i$. This would contradict the primed row $i$ flag condition so we conclude that $m \leq \mu_i$. This means that $\iota(P)$ satisfies the nonprimed row $i$ flag condition. Consequently, $\iota(P)$ satisfies all flag conditions. This shows that $\iota(P) \in \mathrm{PFT}(\lambda/\mu)$.

It is clear that for all $P \in \mathrm{PFT}(\lambda/\mu)$ we have $\iota^2(P)=P$ and that $\ell w(P) +rw(P)=\ell w(\iota(P)) +rw(\iota(P))$ and that $(-1)^{\#(P)}=-(-1)^{\#(\iota(P))}$. This proves the lemma.
\end{proof}

\begin{example}
The involution $\iota$ associates the following two tableaux.
\begin{eqnarray*}
\Yboxdim{14pt}
\young(*****544{{3'}}{{5'}},****433{{3'}}{{6'}}{{8'}},****32{{\textcolor{red}{2}}}{{3'}}{{6'}},***32{{2'}}{{3'}}{{4'}})
\,\,\,\,\, \longleftrightarrow \,\,\,\,\,
\young(*****544{{3'}}{{5'}},****433{{3'}}{{6'}}{{8'}},****32{{\textcolor{red}{2'}}}{{3'}}{{6'}},***32{{2'}}{{3'}}{{4'}})
\end{eqnarray*}
In this case, $m=2$, $i=3$ and $j=7$.
\end{example}

\begin{proof}[Proof of Main Fact \ref{mf2}]
Let $\mu \subset \lambda$ be partitions with exactly $\ell$ rows. There is a canonical bijection:
\begin{eqnarray*}
\bigcup_{\mu \subseteq \rho \subseteq \lambda} \mathrm{OFT}(\rho/\mu) \times \mathrm{UFT}(\lambda/\rho) \rightarrow \mathrm{PFT}(\lambda/\mu)
\end{eqnarray*}
It is given by sending a pair $(P,Q)$ to the tableau obtained by superimposing $P$ and $Q$ on a single Young diagram of shape $\lambda/\mu$. Moreover, if $(P,Q) \rightarrow R$ then $\ell w(R)=wt(P)$ and $rw(R)=wt(Q)$ while $\#(R)=|\rho|-|\mu|$ where $\rho$ is the shape of $P$. Therefore writing $T_{\rho}=\mathrm{OFT}(\rho/\mu) \times \mathrm{UFT}(\lambda/\rho)$, we have:

\begin{eqnarray*}
\sum_{\mu \subseteq \rho \subseteq \lambda} \sum_{(P,Q) \in T_{\rho}}(-1)^{|\rho|-|\mu|}\mathbf{z}^{wt(P)}\mathbf{z}^{wt(Q)}=\sum_{R \in \mathrm{PFT} (\lambda/\mu)} (-1)^{\#(R)}\mathbf{z}^{\ell w(R)}\mathbf{z}^{rw(R)}
\end{eqnarray*}
But the latter is $0$ by lemma \ref{z}. This implies the statement of Main Fact \ref{mf2}.
\end{proof}

\end{document}